\documentclass[a4paper,11pt]{article}
\usepackage[a4paper,left=3cm,right=3cm,top=3cm,bottom=3cm]{geometry}
\usepackage{amsmath, amssymb, amsthm}
\numberwithin{equation}{section} 

\usepackage{newtxtext,newtxmath}
\usepackage{mathtools}
\usepackage{physics}
\usepackage{esint}
\usepackage{subcaption}
\usepackage{graphicx}
\usepackage{enumitem}
\usepackage{xcolor}
\usepackage{hyperref}
\usepackage{cleveref}
\usepackage{comment}
\usepackage{orcidlink}
\usepackage{needspace}

\usepackage[numbers]{natbib}

\title{Addendum to\\
Spectral bounds for the operator pencil\\ of an elliptic system in an angle}
\author{
Michael Tsopanopoulos\, \orcidlink{0009-0007-3167-6862}\thanks{Weierstrass Institute for Applied Analysis and Stochastics, Anton-Wilhelm-Amo-Str. 39, 10117 Berlin, Germany.\\
\texttt{tsopanopoulos@wias-berlin.de}}
}

\newtheorem{theorem}{Theorem}[section]
\newtheorem{lemma}[theorem]{Lemma}
\newtheorem{proposition}[theorem]{Proposition}
\newtheorem{corollary}[theorem]{Corollary}
\theoremstyle{definition}

\theoremstyle{remark}
\newtheorem*{remark}{\textbf{Remark}}

\newcommand{\Mat}{\operatorname{Mat}}
\newcommand{\Id}{\operatorname{Id}}

\begin{document}

\maketitle

\section{Summary}
This addendum complements \cite{Tsopanopoulos2025} and uses the notation and conventions introduced there. In \cite{Tsopanopoulos2025}, the model problem in a plane angle $\mathcal{K}_\alpha$ was studied for solutions of the form
\begin{align*}
    u_\lambda = r^\lambda v ,
\end{align*}
subject to Dirichlet, mixed, and Neumann boundary conditions. For each of these boundary conditions, the existence of a nontrivial solution was reduced to a finite-dimensional spectral condition
\begin{align*}
    0\in \sigma(M_{\lambda,\alpha}),
\end{align*}
where the matrix $M_{\lambda,\alpha}$ depends on the boundary condition, the opening angle $\alpha$, and the standard root of the elliptic tuple. In Sections 6 and 7 of \cite{Tsopanopoulos2025}, numerical range methods and accretive-operator theory were used to derive lower bounds on $|\Re\lambda|$ for Dirichlet and mixed boundary conditions. For Neumann boundary conditions, however, no such bounds were proved within the framework; instead, the discussion was deferred to the existing literature.

The purpose of this addendum is to close this gap. We show that, under the formal positivity assumption on the elliptic tuple, the Neumann problem satisfies the same lower bounds as the Dirichlet problem, apart from the constant solutions corresponding to $\lambda=0$. More precisely, for nontrivial Neumann solutions $r^\lambda v$ with $\lambda\neq 0$, we prove that
\begin{align*}
    |\Re\lambda|\geq 1 \quad \text{for } 0<\alpha\leq \pi\quad \text{and}\quad |\Re\lambda|\geq \frac12 \quad \text{for } \pi<\alpha\leq 2\pi.
\end{align*}

\begin{figure}[h!]
    \centering
    \includegraphics[width=0.6\linewidth]{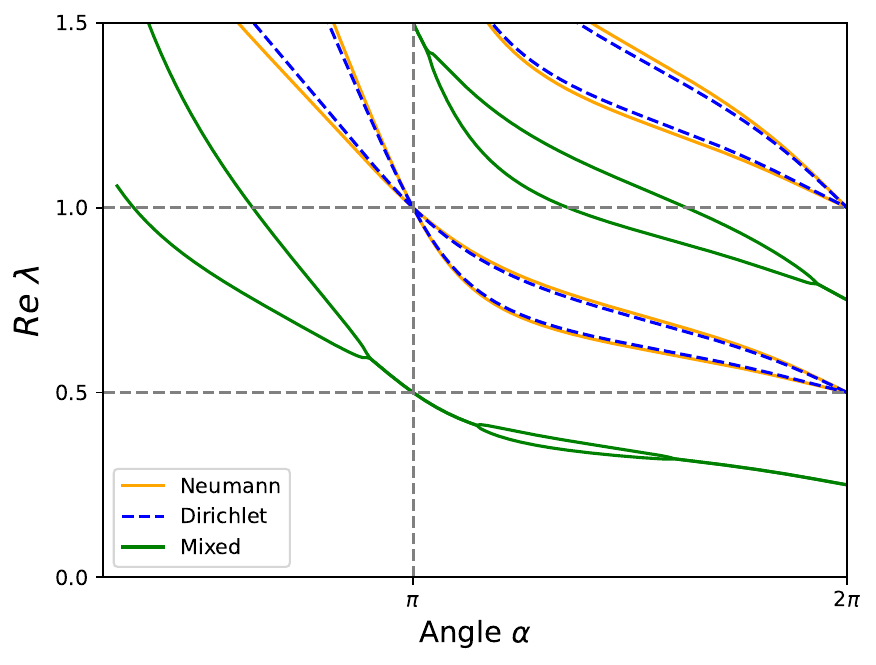}
    \caption{\small Numerically computed branches of $\Re \lambda$ as a function of $\alpha\in [1,2\pi]$ for different boundary conditions. The elliptic tuple is defined by $A_{11}=\begin{pmatrix}
        5&0.6\\0.6&1.5
    \end{pmatrix}$, $A_{12}=\begin{pmatrix}
        0.25&-0.4\\-0.4&-0.2
    \end{pmatrix}$, $A_{22}=\begin{pmatrix}
        1&0\\0&1
    \end{pmatrix}$. The branches for Dirichlet and Neumann boundary conditions are very close to each other.}
    \label{fig:1}
\end{figure}

Thus the framework developed in \cite{Tsopanopoulos2025} can handle Dirichlet, mixed, and Neumann boundary conditions. This is conceptually useful because, in other approaches, such as \cite{VMR2001SPCS}, the treatments of Dirichlet and Neumann boundary conditions are rather different (and mixed boundary conditions are not treated).

The addendum also gives a cleaner presentation of the matrix analysis used in \cite{Tsopanopoulos2025}. A key ingredient is the fractional-power implication
$$
W(F)\subset \operatorname{UHP}
\quad\Longrightarrow\quad
W(F^\lambda)\subset \operatorname{UHP},
\qquad 0<\lambda\leq 1.
$$
In \cite{Tsopanopoulos2025}, this was obtained indirectly from results on accretive operators. Here we give a direct proof based on the Cayley transform and von Neumann's inequality using ideas found in \cite{Haas2003FCSO}. Section \ref{auxi} collects the required auxiliary material, Section \ref{impro} proves the fractional-power estimates, and Section \ref{revdbc} revisits the matrix argument behind the Dirichlet bound in a form suited to the Neumann case. Section \ref{revellcon} records the ellipticity conditions needed in the main proof, and Section \ref{revneumannbound} proves the Neumann bounds.

\section{Auxiliary material for the numerical range}\label{auxi}
Denote by
\begin{align*}
    C(z)\coloneqq (z-i)(z+i)^{-1}
\end{align*}
the Cayley transform. It maps $\operatorname{UHP}=\{z\in \mathbb{C}: \Im z>0\}$ biholomorphically to the unit disk $\mathbb{D}=\{z\in \mathbb{C}: |z|<1\} $. We shall also use the function
\begin{align*}
    \varphi_\lambda(z) = z^\lambda = \exp(\lambda\log(z)),
\end{align*}
where $\log(z) = \log(|z|)+i \arg(z)$ denotes the principal branch with $\arg(z)\in (-\pi,\pi]$. Since the spectrum of the matrices below will be contained in $\operatorname{UHP}$, this branch is well-defined for the corresponding functional calculus. Observe that
\begin{align}\label{ganar}
    \operatorname{sgn}(\lambda)~\varphi_\lambda(\operatorname{UHP})
    &\subset \operatorname{UHP},
    \qquad \lambda\in [-1,1]\setminus \{0\},\\ \nonumber
    \varphi_{i\lambda}(\operatorname{UHP})
    &\subset \mathbb D,
    \qquad \quad~ \lambda>0.
\end{align}

We have the following characterisation for matrices $F$ which satisfy $W(F)\subset \operatorname{UHP}$.  Compare with Prop. B.20 (iv) in \cite{Haas2003FCSO} which characterizes accretive operators, that is, $W(F)\subset \operatorname{RHP}$.

\begin{lemma}\label{cotra}
    Let $F\in \Mat_\ell(\mathbb{C})$. Then $W(F)\subset \operatorname{UHP}$ if and only if $C(F)$ is a contraction, that is, $\|C(F)\|<1$.
\end{lemma}
\begin{proof}
Note that $F+i\Id_\ell$ is invertible since $\sigma(F)\subset W(F)\subset \operatorname{UHP}$. Write $y=(F+i\Id_\ell)x$ and observe
    \begin{align*}
        \|C(F)y\|^2-\|y\|^2 = \|(F-i\Id_\ell)x\|^2 - \|(F+i\Id_\ell)x\|^2 = -4 \Im \langle Fx,x\rangle.
    \end{align*}
Thus, $\Im\langle Fx,x\rangle>0$ for all $x\neq 0$ if and only if $\|C(F)y\|<\|y\|$ for all $y\neq 0$. In finite dimension this is equivalent to $\|C(F)\|<1$.
\end{proof}

We will use von Neumann's inequality which we recall for the convenience of the reader.

\begin{theorem}[von Neumann]\label{vneu}
Let $T:H\to H$ be a bounded linear operator on a Hilbert space $H$ with $\|T\|\leq 1$. Then, for every polynomial $p$,
$$
\|p(T)\| \leq \sup_{|z|\leq 1} |p(z)|\coloneqq \|p\|_\infty.
$$
\end{theorem}

The following consequence is the form in which we use von Neumann's inequality.

\begin{corollary}\label{nagut}
Let $T:H\to H$ be a bounded linear operator with $\|T\|<1$. Let $h:\mathbb{D}\to \mathbb{D}$ be holomorphic. Then
$$
\|h(T)\| <1.
$$
\end{corollary}

\begin{proof}
    Set $r\coloneqq \|T\|<1$. If $r=0$, then $T=0$ and $\|h(T)\|=|h(0)|<1$. Assume now $r>0$. Put $S= T/r$ so $\|S\|\leq 1$. Define $q(z) = h(rz)$. Since $h$ is holomorphic on $\mathbb{D}$, $q$ is holomorphic on a neighbourhood of $\overline{\mathbb{D}}$. Hence, by polynomial approximation and von Neumann's inequality
    \begin{align*}
        \|h(T)\| = \|q(S)\| \leq \sup_{|z|\leq 1}|q(z)| = \sup_{|w|\leq r}|h(w)|<1,
    \end{align*}
    where the last inequality follows since $\{|w|\leq r\}$ is compactly contained in $\mathbb{D}$.
\end{proof}

\section{Revisiting fractional power estimates}\label{impro}
We now prove revised versions of Lemma 6.3 and Lemma 6.4 in \cite{Tsopanopoulos2025}. The ideas are based on the proof of \cite[Prop. 4.13]{Haas2003FCSO}.

\begin{lemma}\label{dopamine}
    Consider $F\in \Mat_\ell(\mathbb{C})$ with $W(F)\subset \operatorname{UHP}$. Then 
    \begin{enumerate}
        \item[i)] $W(F^\lambda)\subset \operatorname{sgn}(\lambda)\operatorname{UHP}$ for $\lambda\in [-1,1]\setminus \{0\}$,
        \item[ii)] $\|F^{i\lambda}\|<1$ for $\lambda>0$,
        \item[iii)] $\operatorname{sgn}(\lambda) (\operatorname{Id}_\ell -  F^{i\lambda} (F^{i\lambda})^* )>0$ for $\lambda\in \mathbb{R}\setminus \{0\}$.
    \end{enumerate}
\end{lemma}

\begin{proof}
    i) First, let $0<\lambda\leq 1$. By Lemma \ref{cotra} it suffices to show that $(C\circ \varphi_\lambda) (F)$ is a contraction. Write
    \begin{align*}
        (C\circ \varphi_\lambda) (F) = (h \circ C) (F),\quad h\coloneqq C \circ \varphi_\lambda \circ C^{-1}.
    \end{align*}
    Note that $h$ is holomorphic and satisfies $h(\mathbb{D})\subset \mathbb{D}$. Moreover, by Lemma \ref{cotra}, $C(F)$ is a contraction. The statement follows by Cor. \ref{nagut}. The statement for $-1\leq \lambda<0$ follows by a similar argument, showing $W(-F^{\lambda})\subset \operatorname{UHP}$ by using (\ref{ganar}) and
    \begin{align*}
        C(-F^{\lambda}) = (h \circ C)(F) \quad \text{for }h = C\circ (-\varphi_{\lambda})\circ C^{-1}.
    \end{align*}

    ii) Let $\lambda>0$. We write $F^{i\lambda} =(h\circ C) (F)$ for $h\coloneqq \varphi_{i\lambda} \circ C^{-1}$. Note that $h(\mathbb{D})\subset \mathbb{D}$. We conclude again by Cor. \ref{nagut}.

    iii) For $\lambda>0$, the statement follows from $
    \| F^{i\lambda} (F^{i\lambda})^* \| = \|F^{i\lambda}\|^2 \stackrel{\text{ii)}}{<}1$ since $ F^{i\lambda} (F^{i\lambda})^*$ is positive semidefinite. For $\lambda<0$, the statement is derived from the case $\lambda>0$ and the fact
    \begin{align*}
        \Id_\ell - B>0 \iff B^{-1}-\Id_\ell>0 \quad \text{for $B>0$}.
    \end{align*}
\end{proof}

\section{Revisiting matrix analysis for Dirichlet boundary conditions}\label{revdbc}
The following arguments will make use of properties \textbf{N1}-\textbf{N8} of the numerical range given in \cite[Section 6.1]{Tsopanopoulos2025}. The next result is a revised version of \cite[Thm. 6.5]{Tsopanopoulos2025}.
\begin{theorem}\label{dirri}
    Consider $F\in \Mat_\ell(\mathbb{C})$ with $W(F)\subset \operatorname{UHP}$, $\lambda\in \mathbb{C}\setminus \{0\}$ with $|\Re \lambda|\leq 1 $, and define $M_\lambda=F^{\lambda}-(F^*)^{\lambda}$. Then $0\notin \sigma(M_\lambda)$.
\end{theorem}
The proof is similar to the original one and we do not repeat all details.

\begin{proof}
Decompose $\lambda$ into real and imaginary parts $\lambda=\lambda_1+i\lambda_2$ and calculate
\begin{align} \label{nase}
        0\in ~\sigma(F^\lambda-(F^*)^\lambda )=\sigma(F^{i\lambda_2}F^{\lambda_1}-(F^*)^{\lambda_1} (F^*)^{i\lambda_2} )
        \iff 0\in~ \sigma(F^{i\lambda_2}F^{\lambda_1} (F^{i\lambda_2})^*-(F^*)^{\lambda_1}  ),
\end{align}
where we use $(F^*)^{-i\lambda_2} = (F^{i\lambda_2})^* $. First, let us assume $\Re \lambda \neq 0$. Then, $0\notin \sigma(M_\lambda)$ follows from (\ref{nase}), \textbf{N3}, \textbf{N5}, and
\begin{align*}
    W(F^{i\lambda_2}F^{\lambda_1} (F^{i\lambda_2})^*) &\subset W'(F^{\lambda_1} ) \subset \operatorname{sgn}(\lambda_1)\operatorname{UHP},\\
    W(-(F^*)^{\lambda_1}) &=-W((F^{\lambda_1})^*) = -\overline{W(F^{\lambda_1})}= -\operatorname{sgn}(\lambda_1)\overline{\operatorname{UHP}} = \operatorname{sgn}(\lambda_1) \operatorname{UHP},
\end{align*}
where we use i) of Lemma \ref{dopamine} and \textbf{N2}, \textbf{N4}, \textbf{N8}. For $\lambda = it$ with $t\in \mathbb{R}\setminus \{0\}$, the result follows from (\ref{nase}) and iii) of Lemma \ref{dopamine}.
\end{proof}

\section{Revisiting ellipticity conditions}\label{revellcon}
We revisit some notions and statements from Appendix A of \cite{Tsopanopoulos2025} concerning ellipticity of a tuple $A = (A_{11},A_{12},A_{22})$. We have the following implications where neither converse implication holds in general:
\begin{align}\label{lukeli}
    \text{A formal positive}&\implies \text{A contractive Neumann well-posed}
        \implies \text{A Neumann well-posed}.
\end{align}
Let $V = (S+i\Id_\ell)D$ with symmetric $S,D\in \Mat_\ell(\mathbb{R})$ and $D>0$ be the standard root of $A$. Then formal positivity is equivalent to $M_A>0$ for
\begin{align*}
        M_A=\begin{pmatrix}
            D(S^2+\Id_\ell)D&-\frac{1}{2}(SD+DS)\\
            -\frac{1}{2}(SD+DS)&\Id_\ell
        \end{pmatrix}.
\end{align*}
Using the Schur complement, this is equivalent to the following matrix being positive definite
\begin{align}\label{boyance}
    P \coloneqq D(S^2+\operatorname{Id})D-\frac{1}{4}(SD+DS)^2.
\end{align}

Contractive Neumann well-posedness was defined by $\rho([D^{-1},SD])< 2$. Let us set
\begin{align}\label{HH}
    H \coloneqq  \frac{i}{2}D^{-1/2}[S,D]D^{-1/2}.
\end{align}
Then contractive Neumann well-posedness is equivalent to $\rho(H)<1$. Indeed, we have that
\begin{align*}
    -2i H = D^{-1/2}[S,D]D^{-1/2} = D^{1/2}[D^{-1},SD]D^{-1/2}.
\end{align*}
Since $[S,D]^T=-[S,D]$, the factor \(i\) makes \(H\) self-adjoint and $\sigma(H)$ is symmetric with respect to $0$. Neumann well-posedness was defined to be $2i\notin \sigma([D^{-1},SD])$ which translates to invertibility of $\Id_\ell - H$. The preceding discussion shows the following 

\begin{lemma}\label{gemam}
Let $A$ be an elliptic tuple with standard root $V=(S+i\Id_\ell)D$. Then:
\begin{enumerate}
    \item $H$ is self-adjoint,
    \item $A$ is Neumann well-posed $\iff $ $1\notin \sigma(H)$,
    \item $A$ is contractive Neumann well-posed $\iff$ $\sigma(H)\subset (-1,1)$,
    \item $A$ is formal positive $\iff$ $P>0$.
\end{enumerate}
\end{lemma}

Thus, we have reformulated all of the ellipticity conditions in (\ref{lukeli}) in terms of $P$ and $H$. Note that an ellipticity condition holds for \(A\) if and only if it holds for its monic reduction \(\widetilde A\); see \cite[Section 3.1]{Tsopanopoulos2025}.

\section{Revisiting Neumann boundary conditions}\label{revneumannbound}
In the following, assume $A$ is Neumann well-posed. We investigate the matrix 
\begin{align*}
    E = \frac{1}{2}D^{-1/2} [S,D]D^{-1/2} + i \Id_\ell = i(\operatorname{Id}_\ell - H)
\end{align*}
appearing in \cite[Prop. 5.4]{Tsopanopoulos2025}. In particular, we derive some properties of 
\begin{align}\label{kortaew}
    T \coloneqq - E^{-1}\overline{E}.
\end{align}

We can write $T$ in terms of $H$ by using the functional calculus:
\begin{align}\label{guhn}
    T =  - (i(\Id_\ell-H))^{-1} (-i(\Id_\ell +H)) =  (\Id_\ell + H ) (\Id_\ell - H)^{-1} .
\end{align}

\begin{lemma}\label{bobosition}
    Let $A$ be a Neumann well-posed elliptic tuple with standard root $V=(S+i\Id_\ell)D$.
    \begin{enumerate}
        \item[i)] $T$ in (\ref{guhn}) is self-adjoint.
        \item[ii)] If $A$ is contractive Neumann well-posed, then $T>0$.
    \end{enumerate}
\end{lemma}
\begin{proof}
    By Lemma \ref{gemam} ii), $(\Id_\ell - H)^{-1}$ and $T$ in (\ref{guhn}) are well-defined. Since $H = H^*$, we get $T^* = T$ by functional calculus. We get $T>0$ from $\rho(H)<1$, using Lemma \ref{gemam} iii).
\end{proof}

Recall that $Z_\alpha$ in (5.4) of \cite{Tsopanopoulos2025} is given by
\begin{align*}
    Z_\alpha=\cos(\alpha)\Id_\ell+D^{1/2}SD^{1/2} \sin(\alpha)+iD\sin(\alpha)
\end{align*}
and put
\begin{align}\label{zege}
Z\coloneqq Z_{\pi/2} = D^{1/2}(S+i\Id_\ell)D^{1/2}.
\end{align}

We can formulate an identity relating formal positivity of $A$ to the numerical range of $ZT$.

\begin{proposition}\label{crucio}
Assume $A$ is Neumann well-posed. Using $S$, $D$, $P$, $E$, $T$ and $Z$ defined above, we have
    \begin{align*}
        P =  (D^{1/2}E)  Q(D^{1/2}E)^*\quad \text{for}\quad Q\coloneqq \frac{1}{2i}(ZT-(ZT)^*).
    \end{align*}
\end{proposition}
In particular, if $P>0$, we have $\frac{1}{2i} (ZT-(ZT)^*)>0$, which translates to $W(ZT)\subset \operatorname{UHP}$ by \textbf{N7}.
\begin{proof}
First, note that 
\begin{align}\label{lbora}
    T E^* =   (\Id_\ell + H ) (\Id_\ell - H)^{-1}  (-i(\Id_\ell - H)) = -i (\Id_\ell + H )=- E^T.
\end{align}
Next, write
\begin{align}\label{kuru}
    (D^{1/2}E)Q(D^{1/2}E)^* = &\frac{1}{2i}( R-R^* ) ,\quad \text{where }R = D^{1/2}EZTE^* D^{1/2}.
\end{align}
Using (\ref{lbora}) and (\ref{zege}), we obtain
\begin{align*}
    R = - D^{1/2}EZ E^T D^{1/2} = - D^{1/2}ED^{1/2} (S+i\Id_\ell)D^{1/2} E^T D^{1/2} 
    = - (X +i D)(S+i\Id_\ell)(-X +i D)
\end{align*} 
for $X=\frac{1}{2}[S,D]$. Note that, using $X^* = -X$ and symmetry of $S,D$,
\begin{align*}
    R^* = (X-iD)(S-i\Id_\ell)(X+iD).
\end{align*}
In (\ref{kuru}), the terms with odd number of $i$ cancel, and we get 
\begin{align*}
    \frac{1}{2i}( R-R^* ) = D^2 + X^2 + DSX -XSD = D(S^2+\operatorname{Id})D-\frac{1}{4}(SD+DS)^2 = P.
    \end{align*}
Here, the second equality is a straightforward expansion.
\end{proof}

Finally, we can formulate and prove the analogous result for Neumann boundary conditions. 

\begin{theorem}\label{main neu}
    Consider a formal positive elliptic tuple $A=(A_{11},A_{12},A_{22})$. Define
    \begin{align*}
        \Lambda_\alpha:=\{\lambda \in \mathbb{C}\setminus \{0\}: \exists\, r^\lambda v\neq 0 \text{ solving (3.3) in \cite{Tsopanopoulos2025} with angle $\alpha$ and Neumann b.c.}\}.
    \end{align*}
    Then:
   \begin{enumerate}[label=\roman*)]
    \item $\Lambda_\alpha \subset \{\lambda \in \mathbb{C}:|\Re \lambda| > 1\}~$ for $~0<\alpha<\pi$.
    \item $\Lambda_\pi =\mathbb{Z}\setminus \{0\}$.
    \item $\Lambda_\alpha \subset \{\lambda \in \mathbb{C}:|\Re \lambda| > \frac{1}{2}\}~$ for $~\pi<\alpha<2\pi$.
    \item $\Lambda_{2\pi} =\frac{1}{2}\mathbb{Z}\setminus \{0\}$.
\end{enumerate}
\end{theorem}

Some arguments in the proof are similar to the ones for the Dirichlet case in \cite[Thm. 7.1]{Tsopanopoulos2025} and are not repeated in full generality.
\begin{proof}
By \cite[Lemma 4.1]{Tsopanopoulos2025} and since the monic reduction preserves formal positivity, it suffices to prove the result for monic formal positive tuples. For such a tuple, let $V=(S+i \Id_\ell)D$ denote its standard root. By Lemma \ref{gemam} iv) and Lemma \ref{bobosition} ii), we have $P>0$ and $T>0$.

By \cite[Prop. 5.4]{Tsopanopoulos2025}, the model problem with Neumann boundary conditions admits a solution $r^\lambda v\neq 0$ for $\lambda \in \mathbb{C}\setminus \{0\}$ and $0<\alpha<2\pi$ if and only if $0\in \sigma(M_{\lambda,\alpha})$ for 
    \begin{align}\label{neumi}
        M_{\lambda,\alpha} =  E~ Z_\alpha^{\lambda_+} E^{-1}- \overline{E} ~\overline{Z_\alpha}^{\lambda_-} \overline{E}^{-1},\quad \text{where }E=i(\Id_\ell - H).
\end{align}
By (\ref{kortaew}) we have $\overline{E} = -ET$, so we can write
\begin{align}\label{neumo}
        M_{\lambda,\alpha} = E ( Z_{\alpha}^{\lambda_+} - T \overline{Z_\alpha}^{\lambda_-} T^{-1} )E^{-1}.
\end{align}
Putting $\tilde{Z}_\alpha \coloneqq T^{-1/2} Z_\alpha T^{1/2}$ and using $\overline{Z_\alpha} = Z_\alpha^*$, $T^* = T$, and the functional calculus, this can be rewritten as
\begin{align}\label{neumi2}
    M_{\lambda,\alpha} = E T^{1/2} ( \tilde{Z}_{\alpha}^{\lambda_+} - (\tilde{Z}_{\alpha}^*)^{\lambda_-} )T^{-1/2} E^{-1}
\end{align}
so that $0\in \sigma(M_{\lambda,\alpha})$ is equivalent to
\begin{align}\label{neumi3}
    0\in \sigma(\tilde{Z}_{\alpha}^{\lambda_+} - (\tilde{Z}_{\alpha}^*)^{\lambda_-} ).
\end{align}

i) $0<\alpha<\pi$.
First, we claim that $W(\tilde{Z}_\alpha)\subset \operatorname{UHP}$ for $0<\alpha<\pi$. Note that 
\begin{align*}
    W(\tilde{Z}_{\alpha}) = W(T^{-1/2} Z_{\alpha}T T^{-1/2} ) = W'( Z_{\alpha}T  ),
\end{align*}
where \textbf{N4} is used. Now $W(Z_\alpha T)\subset \operatorname{UHP}$ follows from $0<\alpha<\pi$, Prop. \ref{crucio} and 
\begin{align*}
    Z_\alpha T-(Z_\alpha T)^* = \sin(\alpha) (ZT-T Z^*).
\end{align*}
This proves $W(\tilde{Z}_\alpha)\subset \operatorname{UHP}$. In particular, $\sigma(\tilde{Z}_\alpha)\subset \operatorname{UHP}$ by \textbf{N3}. Thus, as in the case of Dirichlet boundary conditions, one can replace $\lambda_\pm=\lambda$ in (\ref{neumi3}). The result for i) follows from Theorem \ref{dirri}. 

ii) $\alpha=\pi$. In this case, $Z_\pi=-\Id_\ell$, and the result follows as in the Dirichlet case, since (\ref{neumi}) reads
\begin{align*}
   E~ (-1)^{\lambda_+} E^{-1}- \overline{E} ~(-1)^{\lambda_-} \overline{E}^{-1} = 2i\sin(\lambda \pi).
\end{align*}

iii) $\pi<\alpha<2\pi$.

By the same argument as in the Dirichlet case, there is some $Y_\alpha\in \Mat_\ell(\mathbb{C})$ such that $Y_\alpha = - \tilde{Z}_\alpha^{1/2}$, $W(Y_\alpha)\subset \operatorname{UHP}$, and $\tilde{Z}_{\alpha}^{\lambda_+} - (\tilde{Z}_{\alpha}^*)^{\lambda_-} = Y_{\alpha}^{2\lambda} - (Y_{\alpha}^*)^{2\lambda}$. The result for iii) follows from (\ref{neumi3}) and Theorem \ref{dirri}.

iv)  $\alpha=2\pi$.
The proof is the same limiting argument as in the Dirichlet case and not repeated here.
\end{proof}
\begin{remark}
    Note that \cite[Lemma 3.5]{Tsopanopoulos2025} covers the case $\lambda=0$. In this case, only constant solutions are possible for Neumann boundary conditions which do not occur for Dirichlet or mixed boundary conditions.
\end{remark}

\subsection*{Disclosure of AI Use}
During the preparation and writing of this paper, the author used OpenAI's ChatGPT 5.5 as an AI tool. The tool was used to improve the exposition, identify possible errors, and support the exploration and development of some of the mathematical arguments. The author critically reviewed and verified all AI-assisted contributions and is fully responsible for the final content of this work.

\end{document}